\documentclass[12pt,oneside]{amsart}
\usepackage{amscd,amsmath,amssymb,amsfonts,a4}
\usepackage{pstricks}
\usepackage{color}

\swapnumbers
\theoremstyle{plain}
\newtheorem{thm}{Theorem}[section]

\newtheorem{lem}[thm]{Lemma}
\newtheorem{cor}[thm]{Corollary}

\theoremstyle{definition}

\newtheorem{defn}[thm]{Definition}

\newcommand{\tensor}{\otimes}

\newcommand{\Aut}{{\rm Aut}}

\newcommand{\Char}{{\rm char}}

\renewcommand{\tilde}{\widetilde}

\newcommand{\sO}{{\mathcal O}}

\newcommand{\C}{{\mathbb C}}

\newcommand{\Q}{{\mathbb Q}}

\input{xy}
\xyoption{all}
\begin{document}
\title{Birational Invariance of the $S$-fundamental group scheme}
\author{Amit Hogadi and Vikram Mehta}
\begin{abstract}Let $X$ and $Y$ be nonsingular projective varieties over an algebraically closed field $k$ of positive characteristic. If $X$ and $Y$ are birational, we show that their $S$-fundamental group schemes are isomorphic. 
\end{abstract}
\maketitle
\section{Introduction}

Let $k$ be an algebraically closed field and $X/k$ be a smooth projective variety. When $k=\C$, the Narasimhan-Seshadri theorem \cite{ns} and its generalizations \cite{mr2}, establish an equivalence between unitary representations of the topological fundamental group of $X(\C)$ and polystable vector bundles on $X$ with vanishing Chern classes. While the topological fundamental group does not make sense if $\Char(k)>0$, the category of vector bundles with various other conditions does make sense. This is the main motivation for defining the $S$-fundamental group scheme of a variety as a group scheme associated to a certain Tannakian category of vector bundles. The $S$-fundamental group scheme was first defined in (\cite{bps},$5.1$) for curves and later generalized   to arbitrary dimension in (\cite{langer1},$6.1$). 
\begin{defn}[\cite{langer1},$6.1$]
Let ${\sf Ns}(X)$ be the category of all numerically flat vector bundles on $X$, i.e. vector bundles $E$ such that both $E$ and $E^*$ are nef. Then ${\sf Ns}(X)$ is a Tannakian category (\cite{langer1},$5.4$) with a neutral fiber functor $i_x:{\sf Ns}(X)\to Vect_k$ given by $i_x(E)= E_x$ for a fixed point $x\in X(k)$. The $S$-fundamental group scheme $\pi_1^S(X,x)$ is defined to be $\Aut(i_x)$, the Tannaka group scheme associated to $({\sf Ns}(X),i_x)$. 
\end{defn}

Since one expects $\pi_1^S(X,x)$ to behave like a 'fundamental group', it is reasonable to ask if it actually enjoys properties which are enjoyed by other notions of fundamental group (e.g. topological or etale fundamental group). One such property is birational invariance. It is well known that the etale fundamental group (or even topological or Nori fundamental group) of a smooth projective variety depends only on its birational equivalence class. 
Let $X$ and $Y$ be smooth projective $k$-varieties which are birational. When $\Char(k)=0$, birational invariance of the $S$-fundamental group scheme follows from the birational invariance of the topological fundamental group. In positive characteristic, it was observed by Langer (\cite{langer1},$8.3$) that the the $S$-fundamental group schemes of $X$ and $Y$ are still isomorphic if one assumes $X$ and $Y$ are surfaces. In this case on can reduce to the case when $X$ is the blow up of $Y$ along a smooth center, thanks to the weak factorization theorem (see (\cite{hartshorne},V.$5.3$) or (\cite{wlodarczyk},$12.4$)). However we do not have weak factorization theorems for higher dimensional varieties in characteristic $p$. Thus the proof of the following theorem, which is the main result of this paper, involves a different approach, based on a vanishing theorem (see \eqref{sgen}).
\begin{thm} \label{main} Let $X$ and $Y$ be smooth projective varieties over $k$ and $\phi:X\dashrightarrow Y$ be a rational map which is birational. Let $x_0\in X(k)$ be a point where $\phi$ is defined. Then there exists an isomorphism $ \pi_1^S(X,x_0) \cong \pi_1^S(Y,\phi(x_0))$.
\end{thm}

We now briefly outline the approach to the proof of the above theorem. Consider the diagram 
$$\xymatrix{
	&	Z \ar[rd]^{p}\ar[ld]_{q} & \\
	X\ar@{-->}[rr]^{\phi} & & Y
}$$
where $Z$ is the normalization of the closure of graph of $\phi$. Proving \eqref{main} quickly boils down to showing that $E'=q_*p^*E$ is a numerically flat bundle on $X$ whenever $E$ is a numerically flat bundle on $Y$. The main ingredient for showing this is a vanishing theorem for numerically flat bundles (see \eqref{sgen}) whose immediate corollary is that the sheaf $E'$ is reflexive. We then proceed to show that $E'$ is strongly semistable and has trivial Chern classes which implies that $E'$ is numerically flat by (\cite{langer2},$2.2$).

We would like to mention that H\'el\`ene Esnault has communicated to us that she has a different approach to proving \eqref{main} when the birational map $\phi$ is a morphism. \\

\noindent {\it Acknowledgements}: We thank Adrian Langer for going through this article in detail and for his useful suggestions. We also thank  H\'el\`ene Esnault, Najmuddin Fakhruddin and Yujiro Kawamata for their useful comments.

\section{A vanishing theorem}

Throughout this section, let $k$ be any field of positive characteristic. The main result of this section is the following vanishing theorem, which is central to the proof of Theorem(\ref{main}). In the theorem below, if one takes $X$ to be a smooth projective surface and $V=\sO_X$ and we get Szpiro's theorem (\cite{szpiro},Prop~$2.1$). 
\begin{thm}\label{sgen}
Let $X/k$ be a normal projective variety of dimension at least $2$ and $V$ be a numerically flat vector bundle on $X$. Let $H$ be a globally generated line bundle on $X$ of Kodaira dimension at least two. Then for all sufficiently large integers $n$, $H^1(X,H^{-n}\tensor V)=0$. 
\end{thm}

The following is an immediate corollary of \eqref{sgen}.
\begin{cor}\label{d2}
 Let $p:Z\to Y$ be a dominant morphism of normal projective varieties over $k$. Assume $p_*(\sO_Z)=\sO_Y$. Let $V$ be a numerically flat vector bundle on $Z$. Then $p_*(V)$ has depth at least $2$ at all points of $Y$ of codimension at least $2$. 
\end{cor}
\begin{proof} We may of course assume that dimension of $Y$ is at least two. The theorem is now equivalent to proving that given an ample line bundle $H$ on $Y$, $H^1(Y, H^{-n}\tensor p_*V)=0$ for all $n$ large enough. By the Leray spectral sequence and projection formula, we have
$$ 0 \to H^1(Y, H^{-n}\tensor p_* V) \to H^1(Z, (p^*H)^{-n} \tensor V)$$
The theorem now follows by Theorem \ref{sgen}.
\end{proof}

The following lemma has been proved in (\cite{langer1}, Theorem $9.2$) when $X$ is nonsingular. However, the only role of nonsingularity in the proof is the following boundedness result:  $V$ together with all its Frobenius pull backs are contained in a bounded family. However, it was pointed out to us by Adrian Langer that this boundedness result remains true for arbitrary normal projective $X$ and can be seen by using (\cite{langer-survey},$3.6$) to reduce to the case when $X$ is a normal surface, and then using standard arguments using desingularization, Riemann Roch and Fulton Chern classes for singular varieties.

Thus we have
\begin{lem}{\rm (\cite{langer1}, Theorem $9.2$))}. \label{normall} Let $X$ be a normal projective variety. Let $V$ be a numerically flat bundle on $X$ such that $V$ and all its Frobenius pull backs are contained in a bounded family. Let $A$ be a fixed polarization on $X$. Then for any ample divisor $D$ such that 
$$ DA^{d-1}>\frac{\mu_{max}(\Omega_{X/k})}{p}$$
$H^1(X,V\tensor \sO_X(-D))=0$. 
\end{lem}

We now recall the following lemma from \cite{kawamata}, which will be required in the proof of Theorem \ref{sgen}. 
\begin{lem}[\cite{kawamata},$17$]\label{ktrick}
Let $k$ be any field and $X/k$ be a nonsingular quasiprojective variety. Let $E=\sum_{i=1}^r a_iE_i$ be a snc divisor on $X$, where $a_i's$ are positive integers coprime to $\Char(k)$. Then there exists a finite dominant morphism $f:Y\to X$ and a positive integer $m$, such that 
\begin{enumerate} 
 \item degree($f$) is coprime to $\Char(k)$.
 \item $f^*E_i = mF_i$ where $F=\sum_i F_i$ is a reduced snc divisor.
\end{enumerate}
\end{lem}

\begin{proof}[Proof of \eqref{sgen}] We prove the theorem in five steps.\\
\noindent \underline{Step (1)}: We first reduce to the case where $X$ is a surface. Let $\mu_{max}(\Omega_{X/k})$ denote the slope of the maximal destabilizing subsheaf of $\Omega_{X/k}$, where all slopes are measured with respect to a fixed polarization, $A$, on $X$. Let $D$ be a sufficiently general very ample hypersurface on $X$ such that 
$$ DA^{d-1}>\frac{\mu_{max}(\Omega_{X/k})}{p}$$
Since $D$ is sufficiently general, it is again a normal variety. Since $H$ is a nef line bundle, for every positive integer $n$, $H^{n}\tensor \sO_X(D)$ is ample and $$H^1(X,H^{-n}\tensor V(-D))=0 \ \ \forall \ n\geq 0$$ by \eqref{normall}. Consider the exact sequence 
$$ 0 \to \sO_X(-D) \to \sO_X \to \sO_D \to 0 $$
Tensoring by $H^{-n}\tensor V$, taking cohomology, and using $H^1(X,H^{-n}\tensor V(-D))=0$ we get 
$$ 0 \to H^1(X, H^{-n}\tensor V) \to H^1(D, \left(H^{-n}\tensor V\right)_{|D}) \ \ \ \forall \ n\geq 0$$
Thus, by successively cutting down by general sufficiently ample hypersurfaces, we reduce to the case where $X$ is a normal projective surface. \\

\noindent \underline{Step (2)}: Let $\pi:\tilde{X}\to X$ be a resolution of singularities of $X$. By Leray spectral sequence $H^1(X, H^{-n}\tensor V)$ injects into $H^1(\tilde{X}, \pi^*(H)^{-n}\tensor \pi^*V)$. Moreover, $\pi^*H$ and $\pi^*V$ also satisfy the hypothesis of the theorem. Thus by replacing $X$ by $\tilde{X}$, we may assume $X$ is in fact a nonsingular projective surface. \\

\noindent \underline{Step(3)}: Since $H$ is a nef and big divisor on $X$, there exists an effective $\Q$-divisor $E$ such that $H-E$ is ample. Let $$p:X'\to X$$ be a projective birational morphism such that the union of the exceptional divisor and the proper transform of $E$ is simple normal crossing in $X'$. There exists an effective exceptional $\Q$-divisor $F$ such that $p^*(H-E)-F$ is ample. As in Step (2), we may replace $X$ by $X'$, $H$ by $p^*H$ and $V$ by $p^*V$ to prove the theorem. Also, replace $E$ by $(p^*E+F)$. Thus we reduce to the case where there exists an effective simple normal crossing divisor 
$$\label{aidef}E=\sum a_i E_i$$ such that $H-E$ is ample. \\

\noindent \underline{Step(4)}: In this step, we reduce to the case where the coefficients $a_i's$ occuring in the previous step are integers coprime to $p$. Since ampleness is an open condition
$$ p^m(H-E) + \sum E_i$$
will remain ample for all sufficiently large integer $m$. Clearly, $p^ma_i-1$ is coprime to $p$ for every $i$. In order to prove the theorem, we may replace $H$ by $p^mH$, thanks to (\cite{szpiro},$2.1$). Thus, by replacing $H$ by $p^mH$ and $a_i$ by $(p^ma_i-1)$ and henceforth assume all the $a_i's$ are coprime to $p$.\\

\noindent \underline{Step(5)} Let $f:Y\to X$, be a finite dominant morphism corresponding to the divisor $E=\sum a_iE_i$ as guaranteed by Lemma (\ref{ktrick}). Let $m,F$ be as in the statement of Lemma (\ref{ktrick}). We first claim that $(f^*H)^n(-F)$ is ample divisor on $Y$ for all $n\geq 0$. This can be seen by rewriting
$$ nf^*H - F = \frac{1}{m}f^*(H-E)) + (n-\frac{1}{m})f^*H$$
The divisor on the RHS is a sum of an ample and a nef divisor and thus is ample. \\

\noindent We now fix a polarization $A$ on $Y$ and choose an integer $n$ large enough, such that $L=f^*H^n(-F)$ satisfies
$$ LA > \frac{\mu_{max}(\Omega_{Y/k})}{p}$$
and therefore by Lemma \ref{normall} $H^1(Y, L^{-1}\tensor f^*V)=0$. 

\noindent Consider the short exact sequence $$ 0 \to \sO_Y(-F) \to \sO_Y \to \sO_F \to 0$$ Tensoring by $L^{-1}\tensor f^*V$ we get 
$$ H^0(F, \left(L^{-1}\tensor f^*V\right)_{|F}) \to H^1(Y, f^*H^{-n} \tensor f^*V) \to H^1(Y, L^{-1}\tensor f^*V)(=0)$$
Since $F$ is a reduced curve and $L^{-1}\tensor V$ is an ample vector bundle, $$H^0(F,\left(L^{-1}\tensor f^*V\right)_{|F})=0$$ Therefore 
$$H^1(Y, f^*H^{-n} \tensor f^*V)=0$$ Now to finish the proof we simply observe that $f:Y\to X$ is a finite morphism of degree coprime to $\Char(k)$, therefore $H^1(X,H^{-n}\tensor V)$ is a direct summand of $H^1(Y, f^*H^{-n} \tensor f^*V)$.
\end{proof}

\section{Proof of the Main theorem}

\begin{lem}\label{fpullback}
Let $f:Y\to X$ be a birational morphism of projective varieties where $Y$ is normal and $X$ is nonsingular. Let $F_X$ and $F_Y$ denote the absolute Frobenius morphisms of $X$ and $Y$ respectively. Let $V$ be a vector bundle on $Y$ such that $f_*V$ and $f_*(F_Y^*V)$ are reflexive sheaves. Then the natural map 
$$ \phi: F_X^*(f_*V) \to f_*(F_Y^*V)$$
is an isomorphism.
\end{lem}
\begin{proof}
Since $X$ is smooth $F_X$ is a flat morphism. Therefore $F_X^*(f_*V)$ is also a reflexive sheaf. Since $f$ is a birational morphism, the map $\phi$ is an isomorphism on the complement of the exceptional locus in $X$, which has codimension at least two. But both the sheaves are reflexive and hence it is an isomorphism. 
\end{proof}

The following lemma follows from the arguments given in the proof of (\cite{langer1},Lemma $8.3$). 

\begin{lem}\label{surfacev}
Let $f:Y\to X$ be a birational map of smooth projective surfaces. Then for a numerically flat bundle $V$ on $Y$, $f_*V$ is numerically flat and the natural map $$ f^*f_*V \to V$$ is an isomorphism. 
\end{lem}

\begin{proof}[Proof of \ref{main}]
Let $Z$ be the normalization of the closure of the graph of the rational map $\phi$ in $X\times Y$. Thus we have a diagram 
$$\xymatrix{
& Z \ar[dl]_p \ar[dr]^q &  \\
X \ar@{-->}[rr]^{\phi} & & Y 
}$$

We first claim that for any $V\in {\sf Ns}(X)$, $q_*p^*V$ is a numerically flat vector bundle on $Y$. The family of sheaves 
$$\{q_*p^*V\}_{V\in {\sf Ns}(X)}$$
is a family of reflexive sheaves on $Y$ by Lemma \eqref{d2} which is closed under Frobenius pull backs by Lemma \eqref{fpullback}. Thus in order to prove $q_*p^*V$ is numerically flat, it is enough to show that it is strongly semistable and for any ample divisor $A$, 
$$ c_1(q_*p^*V)A^{d-1}=c_2(q_*p^*V)A^{d-2} =0 $$
Equivalently, by Mehta-Ramanathan restriction theorem, it is enough to show that its restriction of $W=q_*p^*V$ to a sufficiently general complete intersection surface $S\subset Y$ is numerically flat. Since $Z\to X$ is birational, we can find a nonsingular projective surface $T$ and a proper birational morphism $\pi:T\to S$ which fits in the following commutative diagram
$$\xymatrix{
T \ar[d]_{\pi} \ar[r]^h & Z \ar[d]^q \\
S \ar[r]^i & Y
}$$
Let $W=p^*V$. We claim that the map $i^*q_*W \to \pi_*h^*W$ is an isomorphism. To see this we simply observe that both are reflexive sheaves (see \eqref{d2}) and the map is an isomorphism outside codimension two. However by Lemma \eqref{surfacev}, $i^*q_*W$ is numerically flat. \\

Thus the functors $q_*p^*$ gives an equivalence of the Tannakian category ${\sf Ns}(X)$ with ${\sf Ns}(Y)$ with inverse given by $p_*q^*$. 
\end{proof}


\vspace{3mm}
{\small 
\noindent School of Maths, Tata Institute of Fundamental Research, Homi Bhabha Road, Colaba, Mumbai 400005. India. \\
\noindent {\it email}: amit@math.tifr.res.in \\ 

\noindent School of Maths, Tata Institute of Fundamental Research, Homi Bhabha Road, Colaba, Mumbai 400005. India. \\
\noindent {\it email}: vikram@math.tifr.res.in \\
}

\end{document}